\documentclass{amsart}

\usepackage{amssymb}

\usepackage{amsfonts}
\usepackage{mathrsfs}
\usepackage{bbm}

\usepackage[dvips]{graphicx}
\usepackage{tikz}
\tikzstyle{nodo}=[circle,draw,fill,inner sep=0pt, minimum size=0.5*width("k")]
\tikzstyle{infinito}=[circle,inner sep=0pt,minimum size=0mm]
\usepackage[hang,small,labelsep=period,figurename=Fig.,singlelinecheck=false]{caption} 

\usepackage{framed}
\usepackage{mdframed}
\usepackage{tcolorbox}
\tcbuselibrary{raster,skins,listings,skins,breakable,theorems,xparse}

\theoremstyle{plain}
    \newtheorem{theorem}{Theorem}[section]
    \newtheorem{lemma}[theorem]{Lemma}
    
    \newtheorem{corollary}[theorem]{Corollary}
\theoremstyle{definition}
    \newtheorem{definition}{Definition}[section]
    \newtheorem{remark}{Remark}[section]
    \newtheorem*{acknowledgement}{Acknowledgement}
\theoremstyle{remark}
    
\numberwithin{equation}{section}

\newcommand{\cleq}{\lesssim}
\newcommand{\cgeq}{\gtrsim}
\def\norm#1{\left\Vert #1 \right\Vert} 
\def\tbra#1#2{\left\langle #1 , #2 \right\rangle} 

\newcommand{\C}{\mathbb{C}}

\newcommand{\R}{\mathbb{R}}

\DeclareMathOperator{\im}{Im}
\DeclareMathOperator{\re}{Re}

\DeclareMathOperator{\rank}{rank}

\DeclareMathOperator{\slim}{s-lim}

\DeclareMathOperator{\diag}{diag}

\begin{document}

\title[Failure of scattering for NLS on star graph]{Failure of scattering to standing waves for a Schr\"{o}dinger equation with long-range nonlinearity on star graph}
\author[K. Aoki]{Kazuki Aoki}
\address{Department of Mathematics, Graduate School of Science, Osaka University, Toyonaka, Osaka 560-0043, Japan}
\email{k-aoki@cr.math.sci.osaka-u.ac.jp}
\author[T. Inui]{Takahisa Inui}
\address{Department of Mathematics, Graduate School of Science, Osaka University, Toyonaka, Osaka 560-0043, Japan}
\email{inui@math.sci.osaka-u.ac.jp}
\author[H. Mizutani]{Haruya Mizutani}
\address{Department of Mathematics, Graduate School of Science, Osaka University, Toyonaka, Osaka 560-0043, Japan}
\email{haruya@math.sci.osaka-u.ac.jp}
\date{\today}

\begin{abstract}
We consider the Schr\"{o}dinger equation with power type long-range nonlinearity on star graph. Under a general boundary condition  at the vertex,  including Kirchhoff, Dirichlet, $\delta$, or $\delta'$ boundary condition, we show that the non-trivial global solution does not scatter to standing waves. Our proof is based on the argument by Murphy and Nakanishi \cite{MuNa19}, who treated the long-range nonlinear Schr\"{o}dinger equation with a general potential in the Euclidean space, in order to consider general boundary conditions. 
\end{abstract}

\maketitle

\tableofcontents


\section{Introduction}
We will consider the following Schr\"{o}dinger equation with the power type nonlinearity on a star graph $\mathcal{G}$. 
\begin{align*}
	i \partial_t u +\Delta_{M} u +\lambda |u|^{p}u=0, \quad  t \in \mathbb{R},\ x \in \mathcal{G},
\end{align*}
where $p>0$, $\lambda = \pm 1$, and $\Delta_{M}$ denotes the Laplacian on the star graph $\mathcal{G}$ having the boundary condition which is the determined by a matrix $M$ at the vertex of $\mathcal{G}$. This equation is recently studied by many researchers \cite{ACFN11,ACFN14,AnGo18,Kai19,GoOh19} from the view point of the stability of the standing waves. (See also references therein.) Our aim in the paper is to consider failure of scattering when $0<p<1$, whose nonlinearity is called long-range. 

The global behavior of the solution to the following long-range nonlinear Schr\"{o}dinger equation on the Euclidean space $\mathbb{R}^d$ is well studied. 
\begin{align*}
\begin{cases}
	i \partial_t u + \Delta u +\lambda |u|^{p}u=0, & (t,x) \in \mathbb{R} \times \mathbb{R}^d,
	\\
	u(0,x)=u_0 (x), & x \in \mathbb{R}^d
\end{cases}
\end{align*}
where $0<p\leq 2/d$, $d \in \mathbb{N}$, and $\lambda =\pm 1$. It is studied that no non-trivial solution can exhibit asymptotically free behavior by Strauss \cite{Str74}, Barab \cite{Bar84}, and Cazenave \cite{Caz03}. They showed that $u_{+}$ must be $0$ if $e^{-it\Delta} u(t)$ goes to $u_{+}$ as $t \to \infty$ in some function space $X$. More precisely, Strauss \cite{Str74} showed such result for $X=L^2$ and $u_+ \in L^2 \cap L^1$ when $0 < p \leq \min\{1,2/d\}$. Barab \cite{Bar84} bridged the gap between $1<p\leq 2$ when $d=1$ under the additional assumption $u_0 \in H^1 \cap H^{0,1}$ and $\lambda=-1$(defocusing) by using the pseudoconformal identity, where $H^{0,1}:=\{ f : (1+|x|^2)^{1/2} f \in L^2\}$ denotes a weighted $L^2$ space. Cazenave \cite{Caz03} removed the assumption $u_+ \in  L^1$, which is assumed by both Strauss \cite{Str74} and Barab \cite{Bar84}, when $0<p\leq 2/d$ and $d\geq 2$. He also discussed the case $d=1$, $1<p\leq 2$, and $\lambda=1$(focusing). He proved no asymptotically free result for $u_0 \in H^1 \cap H^{0,1}$, $X= H^1 \cap H^{0,1}$, and $u_+ \in H^1 \cap H^{0,1}$ in that case. 
(See \cite{TsYa84,Oza91,GiOz93,HaNa98} for the related works.) When $d=1$ and $1<p\leq 2$, the no asymptotically free problem without any additional assumption, namely under only the assumption $X=L^2$ and $u_+,u_0\in L^2$, still remains open. That is why we do not pursue the case of $1\leq p\leq 2$.

Recently, Murphy and Nakanishi \cite{MuNa19} consider the nonlinear Schr\"{o}dinger equation with a potential. 
\begin{align*}
\begin{cases}
	i \partial_t u + \Delta u+Vu +\lambda |u|^{p}u=0, & (t,x) \in \mathbb{R} \times \mathbb{R}^d,
	\\
	u(0,x)=u_0 (x), & x \in \mathbb{R}^d
\end{cases}
\end{align*}
where $0<p \leq \min\{1,2/d\}$ and 
\begin{align*}
	V \in L^{\infty}(\mathbb{R}:X) \text{ and } X= L^{\frac{2}{p}-}(\mathbb{R}^d) + 
	\begin{cases}
	L^{\frac{d}{2}}(\mathbb{R}^d) &\text{ when } d \geq 3,
	\\
	L^{1+}(\mathbb{R}^2) &\text{ when } d =2,
	\\
	\mathcal{M}(\mathbb{R}) &\text{ when } d =1,
	\end{cases}
\end{align*}
$a \pm$ denotes $a \pm \delta$ for sufficiently small $\delta>0$, and $\mathcal{M}(\mathbb{R})$ denotes the Banach space of complex Radon measures with finite variation on $\mathbb{R}$. They proved that the solution does not scatter to the solitary waves when $0<p<\min\{1,2/d\}$, that is, $u_{+}$ must be $0$ if $u_0,u_{+} \in L^2$ and  $\| u(t) -  e^{it\Delta} u_{+} -l(t) \|_{L^2} \to 0$ as $t \to \infty$ for some $l \in L^\infty(\mathbb{R}_{+}: L^2\cap L^q)$ where $1\leq q<2$. The function $l$ can describe the solitary wave if the equation has. It is worth remarking that their result contains the Dirac delta potential when $d=1$. (See \cite{Seg15,MaMuSe17} for the rerated work.) Such situation is similar to our equation on a star graph since it may have the Dirac delta interaction at the vertex. Though their equation is on the full line, our equation is on half-lines with each other's interaction. We treat not only Dirac delta interaction but also general interactions including Kirchhoff, Dirac, $\delta$, and $\delta'$ interactions. This is a main difference between their equation and ours.


\section{Main result}

\subsection{Preliminaries}
Before the main result, we prepare some notations. See \cite{KoSc06} and \cite{GrIg19} for details. A finite graph is a 4-tuple  $(V,\mathcal{I},\mathcal{E},\partial)$, where $V$ is the finite set of the vertices, $\mathcal{I}$ is the finite set of internal edges, $\mathcal{E}$ is the finite set of external edges, and $\partial$ is a map from $\mathcal{I} \cup \mathcal{E}$ to the set of vertices and ordered pairs of two vertices and satisfying $\partial(i)=(v_1,v_2)$ (possibly $v_1=v_2$) for $i \in \mathcal{I}$ and $\partial(e) = v$ for $e \in \mathcal{E}$.
An element in $\mathcal{I} \cup \mathcal{E}$ is called an edge. We call $v_1=:\partial^{-}(i)$ and $v_2=:\partial^{+}(i)$ initial and final vertex of the internal edge $i \in \mathcal{I}$, respectively. We endow the graph with the metric structure. We assume that for any internal edge $i \in I$ there exist $a_i >0$ and a map $i\mapsto [0,a_i]$ corresponding $\partial^{-}(i)$ to $0$ and $\partial^{+}(i)$ to $a_i$ and that for any external edge $e \in \mathcal{E}$ there exists a map $e \mapsto [0,\infty)$. We call $a_i$ the length of the internal edge $i \in \mathcal{I}$. The graph endow with such metric structure is called metric graph. For given $n \in \mathbb{N}$, a star-shaped metric graph with $n$-edges or simply star graph is a metric graph $ (\{0\}, \emptyset, \{e_j\}_{j=1}^{n}, \partial:\{e_j\}_{j=1}^{n} \to \{0\})$. See the figure below for typical examples. Figure \ref{fig1} or \ref{fig2} is a star graph with 3-edges or 5-edges, respectively. 

\begin{figure}[htb]
\begin{minipage}{0.4\textwidth}
\begin{tikzpicture}[xscale= 0.5,yscale=0.5]
\node at (0,0) [nodo] (0,0) {};

\draw[->] (0,0)--(-3,2);
\draw[->] (0,0)--(3,2);
\draw[->] (0,0)--(0,-3);
\end{tikzpicture}
\caption{}
\label{fig1}
\end{minipage}
\begin{minipage}{0.4\textwidth}
\begin{tikzpicture}[xscale= 0.5,yscale=0.5]  
\node at (0,0) [nodo] (00) {};

\draw[->] (0,0)--(3,0);
\draw[->] (0,0)--(1,2.4);
\draw[->] (0,0)--(-2.5,1.7);
\draw[->] (0,0)--(-2.5,-1.7);
\draw[->] (0,0)--(1,-2.4);
\end{tikzpicture}
\caption{}
\label{fig2}
\end{minipage}
\end{figure}

Through out the paper, let $\mathcal{G}$ be a star graph. 
A function $f$ on $\mathcal{G}$ is given by a vector $f=(f_1,f_2,...,f_n)^{T}$, where each $f_j$ is a complex-valued function defined on $e_j=[0,\infty)$ and $f^T$ denotes the transposition of $f$. The Lebesgue measure on $\mathcal{G}$ is naturally induced by the Lebesgue measures on the half-lines.  We define the function space $L^2(\mathcal{G})$ as the set of measurable and square-integrable functions on each external edge of $\mathcal{G}$. This means 
\begin{align*}
	L^2(\mathcal{G})=\bigoplus_{j=1}^{n} L^2(e_j)
\end{align*}
whose inner product and norm are defined by
\begin{align*}
	\tbra{f}{g} :=\sum_{j=1}^{n} \tbra{f}{g}_{j} =\sum_{j=1}^{n} \int_{e_j} f_j(x) \overline{g_j(x)} dx, 
	\\
	\norm{f}_{L^2(\mathcal{G})}^2 := \tbra{f}{f} =  \sum_{j=1}^{n} \int_{e_j} |f_j(x)|^2 dx=:\sum_{j=1}^{n}  \norm{f}_{L^2(e_j)}^2, 
\end{align*}
where $f=(f_j)_{j=1,...,n}^{T}$, $g=(g_j)_{j=1,...,n}^{T}$ with $f_j , g_j \in L^2(e_j)$ for each $j=1,...,n$. Then, $L^2(\mathcal{G})$ is the Hilbert space. For $1\leq p \leq \infty$, $L^p(\mathcal{G})$ can be defined similarly, i.e., $f \in L^p(\mathcal{G})$ if $f$ is  the component-wise $L^p$ function. The norms are defined as follows. 
\begin{align*}
	\norm{f}_{L^p(\mathcal{G})}:=
	\begin{cases}
	\left( \sum_{j=1}^{n} \norm{f_j}_{L^p(e_j)}^p\right)^{1/p}, &\text{ if } 1\leq p<\infty,
	\\
	\sup_{1\leq j\leq n} \norm{f_j}_{L^{\infty}(e_j)}, &\text{ if } p=\infty.
	\end{cases} 
\end{align*}
Let $H^m(\mathcal{G})$ for $m=1,2$ be 
\begin{align*}
	H^m(\mathcal{G}):= \bigoplus_{j=1}^{n} H^m(e_j),
\end{align*}
whose norms are defined by 
\begin{align*}
	\norm{f}_{H^m(\mathcal{G})}:=
	\left( \sum_{j=1}^{n} \norm{f_j}_{H^m(e_j)}^2\right)^{1/2}.
\end{align*}
We remark that we do not assume any conditions at the joint point $0$. 
If $f(x)=\sum_{j=1}^{n}f_j(x_j)$, where $f_j$ is defined on $e_j$, then the integral of $f$ on $\mathcal{G}$ is defined by 
\begin{align*}
	\int_{\mathcal{G}} f(x) dx = \sum_{j=1}^{n} \int_{e_j} f_j(x_j) dx_j.
\end{align*}

We introduce the Laplacian on the star graph. Let $A,B$ be complex-valued $n \times n$ matrices satisfying the following two conditions.

\begin{itemize}
\item[(A1)] $n \times(2n)$ matrix $(A, B)$ has maximal rank, i.e. $\rank (A,B)=n$.  
\item[(A2)] $AB^*$ is self-adjoint, i.e., $AB^* = (AB^*)^*$, where $X^*:= \overline{X}^T$ denotes the adjoint of the matrix $X$. 
\end{itemize}
Let $M=(A,B)$. We define the Laplacian $\Delta_{M}$ on the star graph as follows. 
\begin{align*}
	&\mathscr{D}(\Delta_{M}) :=\{f \in D(\mathcal{G}): A f(0) + Bf'(0+)=0\},
	\\
	&\Delta_{M} f=(f''_1,f''_2, ... , f''_n)^{T},
\end{align*}
where $D(\mathcal{G})=  \bigoplus_{j=1}^{n} D_j(e_j)$ and $D_j(e_j)$ is the set of functions $f_j \in L^2(e_j)$ satisfying that $f_j$ and $f'_j$ are absolutely continuous and $f_j'' \in L^2(e_j)$ for $j=1,\cdots,n.$
Under the assumption (A1) and (A2), the Laplacian $\Delta_{M}$ is self-adjoint on $L^2(\mathcal{G})$ (see \cite{KoSc06}). 
In fact, the assumption (A1) and (A2) are equivalent to that the differential operator $\frac{d^2}{dx^2}f=(f''_1,f''_2, ... , f''_n)^{T}$ on a space of test functions has self-adjoint extensions in $L^2(\mathcal{G})$. Under the assumption (A1) and (A2), $e^{it\Delta_{M}}$ can be defined as the unitary operator on $L^2(\mathcal{G})$ by the Stone theorem. Below, we always assume (A1) and (A2) on $M$.

The typical examples of $\Delta_{M}$ are the following. 
\begin{enumerate}
\renewcommand{\theenumi}{\alph{enumi}}
\item Kirchhoff boundary condition: Let $M=(A,B)$ be
\begin{align*}
	A=
	\begin{pmatrix}
	1 & -1 & 0 & \cdots & 0 & 0
	\\
	0 & 1 & -1 & \cdots & 0 & 0
	\\
	\vdots & \vdots & \vdots &   & \vdots & \vdots
	\\
	0 & 0 & 0 & \cdots & 1 & -1
	\\
	0 & 0 & 0 & \cdots & 0 & 0
	\end{pmatrix},
	\quad
	B=
	\begin{pmatrix}
	0 & 0 & 0 & \cdots & 0 & 0
	\\
	0 &0 & 0 & \cdots & 0 & 0
	\\
	\vdots & \vdots & \vdots &   & \vdots & \vdots
	\\
	0 & 0 & 0 & \cdots & 0 & 0
	\\
	1 & 1 & 1 & \cdots & 1 & 1
	\end{pmatrix}.
\end{align*} 
For such $M$, $A f(0) + Bf'(0+)=0$ implies that $f_j(0)=f_k(0)$ for any $j,k \in \{1,2,\cdots,n\}$ and $\sum_{j=1}^{n} f_j'(0+)=0$.  This is called the Kirchhoff boundary condition. We denote the the Laplacian determined by the Kirchhoff boundary condition by $\Delta_{K}$.  In the sense that there is no external force, i.e., no external interaction at the vertex, the Laplacian $\Delta_{K}$ is regarded as free Laplacian on the star graph. 

\item Dirac delta ($\delta$) boundary condition: Let $\alpha \neq 0$ and  $M=(A,B)$ be
\begin{align*}
	A=
	\begin{pmatrix}
	1 & -1 & 0 & \cdots & 0 & 0
	\\
	0 & 1 & -1 & \cdots & 0 & 0
	\\
	\vdots & \vdots & \vdots &   & \vdots & \vdots
	\\
	0 & 0 & 0 & \cdots & 1 & -1
	\\
	-\alpha & 0 & 0 & \cdots & 0 & 0
	\end{pmatrix},
	\quad
	B=
	\begin{pmatrix}
	0 & 0 & 0 & \cdots & 0 & 0
	\\
	0 &0 & 0 & \cdots & 0 & 0
	\\
	\vdots & \vdots & \vdots &   & \vdots & \vdots
	\\
	0 & 0 & 0 & \cdots & 0 & 0
	\\
	1 & 1 & 1 & \cdots & 1 & 1
	\end{pmatrix}.
\end{align*} 
For this $M$, $A f(0) + Bf'(0+)=0$ implies that $f_j(0)=f_k(0)$ for any $j,k \in \{1,2,\cdots,n\}$ and $\sum_{j=1}^{n} f_j'(0+)=\alpha f_k (0)$. This is called the Dirac delta boundary condition. 

\item Dirichlet (zero) boundary condition: Let $M=(A,B)$ be
\begin{align*}
	A=I,
	\quad
	B=0,
\end{align*}
where $I$ is the $n \times n$ identity matrix and $0$ is the zero matrix. 
For this $M$, $A f(0) + Bf'(0+)=0$ implies that $f_j(0)=0$ for any $j \in \{1,2,\cdots,n\}$.  This is so-called the Dirichlet zero (or, simply, Dirichlet) boundary condition. The Dirichlet boundary condition means that star graph is not connected at the vertex. We can regard the star graph as $n$ half lines without each other's interaction. We denote the the Laplacian determined by the Dirichlet boundary condition by $\Delta_{D}$.

\item $\delta'$ boundary condition: Let $\alpha \in \mathbb{R}$ and  $M=(A,B)$ be
\begin{align*}
	A=
	\begin{pmatrix}
	0 & 0 & 0 & \cdots & 0 & 0
	\\
	0 &0 & 0 & \cdots & 0 & 0
	\\
	\vdots & \vdots & \vdots &   & \vdots & \vdots
	\\
	0 & 0 & 0 & \cdots & 0 & 0
	\\
	1 & 1 & 1 & \cdots & 1 & 1
	\end{pmatrix},
	\quad
	B=
	\begin{pmatrix}
	1 & -1 & 0 & \cdots & 0 & 0
	\\
	0 & 1 & -1 & \cdots & 0 & 0
	\\
	\vdots & \vdots & \vdots &   & \vdots & \vdots
	\\
	0 & 0 & 0 & \cdots & 1 & -1
	\\
	-\alpha & 0 & 0 & \cdots & 0 & 0
	\end{pmatrix}.
\end{align*} 
For this $M$, $A f(0) + Bf'(0+)=0$ implies that $f'_j(0)=f'_k(0)$ for any $j,k \in \{1,2,\cdots,n\}$ and $\sum_{j=1}^{n} f_j(0+)=\alpha f'_k (0)$. This is called $\delta'$ boundary condition. 
\end{enumerate}

\begin{remark}
$M=(A,B)$ is not determined uniquely from $-\Delta_{M}$. Indeed, $A=-\alpha I$ and $B=(1)_{1\leq i,j \leq n}$, which is different from (b), also imply the Laplacian with Dirac delta boundary condition. See \cite{KoSc06} for details. 
\end{remark}

We consider the following nonlinear Schr\"{o}dinger equation on the star graph $\mathcal{G}$.
\begin{align}
\tag{NLS}
\label{NLS}
\begin{cases}
i \partial_t u +\Delta_{M} u +\lambda |u|^{p}u=0, & t \in \mathbb{R},\ x \in \mathcal{G},
\\
u(0,x) =u_0(x), &x \in \mathcal{G},
\end{cases}
\end{align}
where $0<p<1$.

\subsection{Main result}

The global existence of $L^2$-solution to \eqref{NLS} is obtained by \cite{GrIg19}. 
We have the following main result.

\begin{theorem}
\label{thm2.1}
Let $0<p<1$, $u_0 \in L^2(\mathcal{G})$. If $u$ is a global solution of \eqref{NLS} satisfying
\begin{align*}
	\norm{u(t) - (e^{it\Delta_{K}} u_{+} +l(t))}_{L^2(\mathcal{G})} \to 0 \quad (t \to \infty)
\end{align*}
for some $u_{+} \in L^2(\mathcal{G})$, 
where $l \in L^\infty((0,\infty): L^2(\mathcal{G}) \cap L^q(\mathcal{G}))$ for some $1\leq q<2$, then $u_{+} \equiv 0$. The similar result holds in the negative time direction. 
\end{theorem}

Since $-\Delta_{K}$ is the free Laplacian on the star graph, this means the failure of scattering to standing waves for the long-range nonlinear Schr\"{o}dinger equation on the star graph. This also means no assympotically free result if $l=0$. We note that we only treat the case of $0<p<1$ and assume only $u_0 \in L^2(\mathcal{G})$ and $v_{+} \in L^2(\mathcal{G})$. 

We also have the following. Let $P_{ac}(M)$ denote the projection onto the absolutely continuous spectral subspace of $L^2(\mathcal{G})$ associated to $-\Delta_{M}$. 

\begin{corollary}
\label{cor2}
Let $0<p<1$. If $u$ is a global solution of \eqref{NLS} satisfying
\begin{align*}
	\norm{u(t) - (e^{it\Delta_{M}} P_{ac}(M) u_{+} +l(t))}_{L^2(\mathcal{G})} \to 0 \quad (t \to \infty)
\end{align*}
where $l \in L^\infty(0,\infty: L^2(\mathcal{G}) \cap L^q(\mathcal{G}))$ for some $1\leq q<2$, then $v_{+} \equiv 0$.
\end{corollary}

This corollary can be proved by Lemma \ref{lem2.3} below.  

\begin{remark}
$-\Delta_{M}$ may have negative eigenvalues. For instance, the Laplacian with Dirac delta boundary condition (b) has a negative eigenvalue if $\alpha<0$. The sufficient and necessary condition for $-\Delta_{M}$ to have negative eigenvalues is known by \cite[Lemma 3.1]{KoSc06}. It is worth emphasizing that $-\Delta_{K}$ has no negative eigenvalues. 
\end{remark}

\begin{remark}
Yoshinaga investigated the Schr\"{o}dinger equation with the short-range nonlinearity on the star graph in \cite{Yos18}. He showed that the solution of \eqref{NLS} with $\lambda=-1$, $p>2$, and $M=K$ behaves like the linear solution at infinite time. More precisely, if  $u_0 \in \Sigma_c$, where $\Sigma_c := \{u \in H^1(\mathcal{G}): \norm{xu}_{L^2(\mathcal{G})}< \infty \text{ and } u_1(0)=\cdots=u_n(0)\}$,  and $u$ denotes the solution to \eqref{NLS} with $\lambda=-1$, $p>2$, and $M=K$, then there exist $u_{\pm} \in L^2(\mathcal{G})$ such that
\begin{align*}
	\lim_{t \to \pm \infty} \norm{u(t)- e^{it\Delta_{K}}u_{\pm}}_{L^2(\mathcal{G})}=0.
\end{align*}
\end{remark}

\subsection{Idea of the proof}
\label{sec2.3}

Our proof is based on the argument by Murphy and Nakanishi \cite{MuNa19}. 
Multiplying the free solution $w=e^{it\Delta}\varphi$ to the nonlinear equation and integrating it on whole space, they obtained the weak formula
\begin{align*}
	i \frac{d}{dt} \tbra{u}{w}+\tbra{Vu}{w} +\lambda \tbra{|u|^p u}{w}=0.
\end{align*}
To estimate the third term, they used the so-called Dollard decomposition $e^{it\Delta}=\mathscr{MDFM}$, where $\mathscr{M}$ is a multiplier operator, $\mathscr{D}$ is a dilation operator, and $\mathscr{F}$ is the usual Fourier transform on the Euclidean space (see \cite{MuNa19} for the detalis). 
The decomposition $e^{it\Delta}=\mathscr{MDFM}$ is also called the factorization formula. 
They also applied the Strichartz estimates and the local well-posedness argument to estimate the second term.

For our equation \eqref{NLS}, $e^{it\Delta_{K}}$ is the free propagator in the sense that the Kirchhoff boundary condition denotes no external force. However, it is not clear that $e^{it\Delta_{K}}$ has a factorization formula. 
That is why we use $w=e^{it\Delta_{D}}\varphi$, whose Dollard decomposition $e^{it\Delta_{D}}=\mathcal{MDFM}$ is used in \cite{EsKaHa19} to analyze a nonlinear Schr\"{o}dinger equation on the half-line with an inhomogeneous Dirichlet boundary, as a test function. 
Multiplying $w=e^{it\Delta_{D}}\varphi$ to \eqref{NLS} and integrating it on $e_j$, we get the following weak formula:
\begin{align*}
	i \frac{d}{dt} \tbra{u}{w}_{j} + u_j(t,0+) \overline{\partial_x w_j (t,0+)} +\lambda \tbra{|u|^p u}{w}_{j}=0,
\end{align*}
where the Dirichlet zero boundary condition also plays a crucial role. Otherwise, a term involving $\partial_x u(t,0+)$ may appear in the weak formula which seems to be difficult to estimate in the $L^2$-framework. We will estimate the third term by the Dollard decomposition of $e^{it\Delta_{D}}$. We will apply the Strichartz estimates of $e^{it\Delta_{M}}$ and the local well-posedness argument of \eqref{NLS}, which are obtained by Grecu and Ignat \cite{GrIg19}, to estimate the second term based on the argument in \cite{MuNa19}.

\begin{lemma}
\label{lem2.2}
Let $0<p<1$. If $u$ is a global solution of \eqref{NLS} satisfying
\begin{align*}
	\norm{u(t) - (e^{it\Delta_{D}} v_+ + l(t))}_{L^2(e_j)} \to 0 \quad (t \to \infty)
\end{align*}
for some $j \in \{1, ... , n\}$ and $v_+ \in L^2(\mathcal{G})$, then $v_+ \equiv 0$ on $e_j$.
\end{lemma}

This seems to be far from Theorem \ref{thm2.1}. To replace $e^{it\Delta_{D}}$ by $e^{it\Delta_{K}}$, we show the following linear asymptotic lemma. 

\begin{lemma}
\label{lem2.3}
For any $u_{+} \in L^2(\mathcal{G})$, there exists $v_+ \in L^2(\mathcal{G})$ such that
\begin{align*}
	\norm{e^{it\Delta_{M}} P_{ac}(M) u_{+} - e^{it\Delta_{D}} v_+}_{L^2(\mathcal{G})} \to 0 \quad (t \to \infty)
\end{align*}
\end{lemma}

Combining Lemma \ref{lem2.2} and Lemma \ref{lem2.3} as $M=K$, we obtain Theorem \ref{thm2.1}. Corollary \ref{cor2} comes from Lemma \ref{lem2.2} and Lemma \ref{lem2.3} as $M$.


\section{Proofs}

\subsection{Proof of Lemma {\ref{lem2.2}}}
\label{sec3.1}

We define the dilation operator $\mathcal{D}$ and multiplier operator $\mathcal{M}$ by
\begin{align*}
	\mathcal{D}f(x)=(\mathcal{D}(t)f)(x)= (2it)^{-\frac{1}{2}} f\left(\frac{x}{2t}\right), \quad
	\mathcal{M}f(x)=(\mathcal{M}(t)f)(x)=e^{\frac{i|x|^2}{4t}} f(x),
\end{align*}
for a function $f$ on $\mathcal{G}$. These operators are invertible. Let $\widetilde{u}=\mathcal{D}^{-1}\mathcal{M}^{-1}u$, $\widetilde{w}=\mathcal{D}^{-1}\mathcal{M}^{-1}w$, and $\widetilde{l}=\mathcal{D}^{-1}\mathcal{M}^{-1}l$. We define the Fourier-Sine transformation $\mathcal{F}=\mathcal{F}_j^{\text{s}}$ by 
\begin{align*}
	\mathcal{F} f (\xi_j)=\mathcal{F}_j^{\text{s}} f (\xi_j) =\frac{1}{i} \sqrt{\frac{2}{\pi}} \int_0^{\infty} \sin (x_j\xi_j) f(x_j) dx_j \text{ for } \xi_j \in e_j.
\end{align*}
We remark that $\mathcal{F}$ is invertible and $\mathcal{F}^{-1}=-\mathcal{F}$. 
For the Fourier-Sine transformation $\mathcal{F}$, the following lemma holds as same as the usual Fourier transform.
\begin{lemma}[Hausdorff--Young inequality]
\label{lem3.1}
We have 
\begin{align*}
	\norm{\mathcal{F}_{j}^{\text{\rm s}} f}_{L^p(e_j)} \leq \norm{f}_{L^{p'}(e_j)}
\end{align*}
for any $p \in [2,\infty]$, where $p'$ denotes the H\"{o}lder conjugate of $p$. 
\end{lemma} 
\begin{proof}
It obviously holds that $\norm{\mathcal{F}_{j}^{\text{s}} f}_{L^\infty(e_j)} \leq \norm{f}_{L^{1}(e_j)}$. Moreover, we have $L^2$-isometry, namely, $\norm{\mathcal{F}_{j}^{\text{s}} f}_{L^2(e_j)} =\norm{f}_{L^{2}(e_j)}$. This can be proved by extending $f$ to the odd function on the real line and $L^2$-isometry of the usual Fourier transform. 
\end{proof}

We use a contradiction argument to show Lemma \ref{lem2.2}. We suppose that $v_+ \not\equiv 0$. 
By multiplying $w=e^{it\Delta_{D}}\varphi$ such that $\varphi(0)=0$, which is a solution of $i \partial_t w +\Delta_{D} w =0$ with the initial data $\varphi$, which has the Dirichlet boundary condition, to the nonlinear equation \eqref{NLS} and integrating it on an edge $e_j$, we get
\begin{align}
\label{eq3.1}
	i \frac{d}{dt} \tbra{u}{w}_{j} =- u_j(t,0+) \overline{\partial_x w_j (t,0+)} - \tbra{F(u)}{w}_{j},
\end{align}
where we set $F(u)=\lambda |u|^p u$ for simplicity. We take sufficiently smooth and decaying fast $\varphi$.  We will estimate $u_j(t,0+) \overline{\partial_x w_j (t,0+)}$ and $\tbra{F(u)}{w}_{j}$. 
First, we treat $\tbra{F(u)}{w}_{j}$. Since $F$ is gauge invariant, we have 
\begin{align*}
	\tbra{F(u)}{w}_{j} = t^{-\frac{p}{2}} \tbra{F(\widetilde{u})}{\widetilde{w}}_{j}.
\end{align*}
\begin{lemma}
\label{lem3.2}
Let $0<p<1$. If $u$ is a forward-global solution of \eqref{NLS} satisfying
\begin{align*}
	\norm{u(t) - (e^{it\Delta_{D}} v_++l(t))}_{L^2(e_j)} \to 0 \quad (t \to \infty)
\end{align*}
for some $j \in \{1, ... , n\}$ and $v_+ \in L^2(e_j)$, then 
\begin{align*}
	\tbra{F(\widetilde{u})}{\widetilde{w}}_{j} = \tbra{F(\mathcal{F} v_+)}{\mathcal{F}\varphi}_{j} + o(1) \text{ as } t \to \infty.
\end{align*}
\end{lemma}
\begin{proof}
We show the following two estimates.
\begin{align}
\label{eq3.2}
	\widetilde{w} &= \mathcal{F} \varphi + o(1) \text{ in } L^2 (e_j) \cap L^{\infty} (e_j) \text{ as } t \to \infty,
	\\
\label{eq3.3}
	F(\widetilde{u}) &= F(\mathcal{F} v_+) + o(1) \text{ in } L^{\frac{2}{p+1}} (e_j) \cap L^{\frac{2}{p+1}-} (e_j) \text{ as } t \to \infty,
\end{align}
where we write $a\pm$ as $a\pm \varepsilon$ for some sufficiently small $\varepsilon>0$. It is worth emphasizing that $L^{\frac{2}{p+1}-} (e_j)$ will be used to treat $l$. 
First, we show \eqref{eq3.2}.  Since $e^{it\Delta_{D}}=\mathcal{M}\mathcal{D}\mathcal{F}\mathcal{M}$, we have $\widetilde{w}=\mathcal{F}\mathcal{M} \varphi$. Thus, it follows that
\begin{align*}
	\norm{\widetilde{w} - \mathcal{F}\varphi}_{L^{r}(e_j)} 
	=\norm{\mathcal{F} (M-1) \varphi}_{L^{r}(e_j)} 
	\leq \norm{(M-1)\varphi}_{L^{r'}(e_j)}
	\cleq \frac{1}{|t|} \norm{|x|^2 \varphi}_{L^{r'}(e_j)},
\end{align*}
uniformly in $t$ for $2\leq r \leq \infty$. This means \eqref{eq3.2}. 

Next, we prove \eqref{eq3.3}. We have
\begin{align*}
	u-e^{it\Delta_{D}} v_+ -l
	=\mathcal{M}\mathcal{D}(\widetilde{u} - \mathcal{F} v_+ - \widetilde{l}) + \mathcal{M}\mathcal{D}\mathcal{F}(1-\mathcal{M})v_+ \text{ on } e_j.
\end{align*}
Let $I:=u-e^{it\Delta_{D}} v_+ -l$ and $I\!\!I:=\mathcal{M}\mathcal{D}\mathcal{F}(1-\mathcal{M})v_+$. By the assumption, $\norm{I}_{L^2(e_j)} \to 0$ as $t \to \infty$. We also have $\norm{I\!\!I}_{L^2(e_j)}=\norm{(1-\mathcal{M})v_+}_{L^2(e_j)} \to 0$ as $t \to \infty$ by $L^2$-isometry of $D$ and $\mathcal{F}$ and the Lebesgue dominated convergence theorem. Therefore, we obtain $\|\widetilde{u} - \mathcal{F} v_+ - \widetilde{l} \| _{L^2(e_j)} \to 0$ as $t \to \infty$. Moreover, we have $\| \widetilde{l}\|_{L^{q}}= t^{-(1/q-1/2)} \norm{l}_{L^q} \to 0$ for $1<q<2$. Thus, we get the following. 
\begin{align*}
	&\norm{F(\tilde{u}) - F(\mathcal{F}v_+) }_{L^{\frac{2}{p+1}} (e_j) + L^{\frac{2}{p+1}-} (e_j)}
	\\& \quad \cleq ( \| \widetilde{u} - \mathcal{F}v_+-\widetilde{l} \|_{L^2} + \| \widetilde{l}\|_{L^{2-}} ) (\norm{\widetilde{u}}_{L^{2}}^{p} + \norm{\mathcal{F}v_+}_{L^2}^p)
	\\& \quad \to 0
\end{align*}
as $t \to \infty$. Combining \eqref{eq3.2} and \eqref{eq3.3}, we get
\begin{align*}
	&| \tbra{F(\widetilde{u})}{\widetilde{w}}_{j} - \tbra{F(\mathcal{F} v_+)}{\mathcal{F}\varphi}_{j}|
	\\
	&\quad \leq | \tbra{F(\widetilde{u})}{\widetilde{w}}_{j} - \tbra{F(\mathcal{F} v_+)}{\widetilde{w}}_{j}|
	+|  \tbra{F(\mathcal{F} v_+)}{\widetilde{w}}_{j} - \tbra{F(\mathcal{F} v_+)}{\mathcal{F}\varphi}_{j}|
	\\
	& \quad \leq \norm{F(\tilde{u}) - F(\mathcal{F}v_+)}_{L^{\frac{2}{p+1}}  + L^{\frac{2}{p+1}-}} \norm{\widetilde{w}}_{L^{\frac{2}{1-p}} \cup L^{\frac{2}{1-p}+}}
	+\norm{v_+}_{L^2}^{p+1} \norm{\widetilde{w}-\mathcal{F}\varphi}_{L^{\frac{2}{1-p}}} 
	\\
	&\quad  \to 0 \text{ as } t \to \infty.
\end{align*}
\end{proof}

Secondly, we estimate $u_j(t,0+) \overline{\partial_x w_j (t,0+)}$ in \eqref{eq3.1}. 
\begin{lemma}
\label{lem3.3}
If $\varphi \in C_{0}^{\infty}([0,\infty))$ and $\varphi(0)=0$, it holds that
\begin{align*}
	\int_{t_0}^{t_0+T} \left|u_j(t,0+) \overline{\partial_x w_j (t,0+)}\right| dt 
	\cleq T^{\frac{3}{4}} t_0 ^{-\frac{1}{2}} \norm{u_j}_{L^4 (t_0,t_0+T; L^{\infty}(e_j))}.
\end{align*}
\end{lemma}

\begin{proof}
We have
\begin{align*}
	\text{(L.H.S)} 
	\leq T^{\frac{3}{4}} \norm{u_j}_{L^4 (t_0,t_0+T; L^{\infty}(e_j))} \norm{\partial_x w}_{L^{\infty}([t_0,t_0+T]\times e_j)}.
\end{align*}
Since we have
\begin{align*}
	w(t,x_j)= e^{it\Delta_{D}} \varphi= \frac{1}{\sqrt{4\pi i t}} \int_{0}^{\infty} \left( e^{\frac{i|x_j-y_j|^2}{4t}} - e^{\frac{i|x_j+y_j|^2}{4t}}\right) \varphi(y_j) dy_j,
\end{align*}
we get
\begin{align*}
	\partial_x w(t,x_j) =  \frac{1}{\sqrt{4\pi i t}} \int_{0}^{\infty} \left( e^{\frac{i|x_j-y_j|^2}{4t}} + e^{\frac{i|x_j+y_j|^2}{4t}}\right) \varphi'(y_j) dy_j
\end{align*}
and thus $\norm{\partial_x w}_{L^{\infty}([t_0,t_0+T]\times e_j)} \leq t_0^{-\frac{1}{2}} \norm{\varphi'}_{L^1(e_j)}$. 
\end{proof}

By the local well-posedness and $L^2$-conservation law, we get
\begin{align}
\label{eq3.4}
	 \norm{u_j}_{L^4 (t_0,t_0+T; L^{\infty}(e_j))} 
	 \leq \norm{u_j}_{L^4 (t_0,t_0+T; L^{\infty}(\mathcal{G}))} 
	  \cleq \norm{u_0}_{L^2(\mathcal{G})}
\end{align}
for sufficiently small $T>0$ (see \cite[Proof of Theorem B]{GrIg19}). Thus we get the following lemma.

\begin{lemma}
\label{cor3.4}
If $\varphi  \in C_{0}^{\infty}([0,\infty))$ and $\varphi(0)=0$, it is true that
\begin{align*}
	\int_{1}^{\tau} \left|u_j(t,0+) \overline{\partial_x w_j (t,0+)}\right| dt 
	\cleq \tau^{1/2}
\end{align*}
for any $\tau >1$. 
\end{lemma}

\begin{proof}
Let $T \in (0,1)$ be a small real number satisfying \eqref{eq3.4} and $[a]$ denote the integer part of $a$. Then, we have
\begin{align*}
	\int_{1}^{\tau} \left|u_j(t,0+) \overline{\partial_x w_j (t,0+)}\right| dt 
	&\leq \sum_{n=1}^{[\tau-T]+1} \int_{n}^{n+T} \left|u_j(t,0+) \overline{\partial_x w_j (t,0+)}\right| dt.
\end{align*}
It follows from Lemma \ref{lem3.3} and \eqref{eq3.4} that
\begin{align*}
	\cdots &\cleq \sum_{n=1}^{[\tau-T]+1} T^{\frac{3}{4}} n ^{-\frac{1}{2}} \norm{u_j}_{L^4 (n,n+T; L^{\infty}(e_j))}
	\\
	&\cleq T^{\frac{3}{4}} \norm{u_0}_{L^2(\mathcal{G})}  \sum_{n=1}^{[\tau-T]+1}  n ^{-\frac{1}{2}} 
	\\
	&\cleq  T^{\frac{3}{4}} \norm{u_0}_{L^2(\mathcal{G})}   \tau^{1/2}. 
\end{align*}
\end{proof}

\begin{lemma}
\label{rmkA.3}
Let $r \in [2,\infty)$ and $f \in L^r([0,\infty))$. For any $\varepsilon>0$, there exists $\varphi \in \{ f \in C_{0}^{\infty}([0,\infty)): f(0)=0\}$ such that
\begin{align*}
	\norm{f-\mathcal{F}\varphi}_{L^r(0,\infty)} \leq \varepsilon.
\end{align*}
\end{lemma}

\begin{proof}
Let $r':=\frac{r}{r-1} \in (1,2]$. 
The embedding $C_{0,{\rm odd}}^{\infty}(\mathbb{R}) \subset L_{\rm odd}^{r'}(\mathbb{R})$ is dense in $L^{r'}$-topology, where $X_{{\rm odd}}=\{f \in X: f(x)=-f(-x)\}$. Indeed, the density $C_{0}^{\infty}(\mathbb{R}) \subset L^r(\mathbb{R})$ can be shown by a mollifier argument and the mollifier of an odd function $f$ can be also odd. This density implies that the embedding $\{ f \in C_{0}^{\infty}([0,\infty)): f(0)=0\} \subset \{ f \in L^{r'}([0,\infty))\}$ is dense in $L^{r'}$-topology. By the Hausdorff--Young inequality, Lemma \ref{lem3.1}, $\mathcal{F}^{-1}f$ belongs to $L^{r'}([0,\infty))$.

for $r \geq 2$, $f \in L^{r}$, and $\varepsilon>0$, there exists $\varphi \in \{ f \in C_{0}^{\infty}([0,\infty)): f(0)=0\}$ such that
\begin{align*}
	\norm{f-\mathcal{F}\varphi}_{L^r} \leq \norm{\mathcal{F}^{-1}f - \varphi}_{L^{r'}} \leq \varepsilon.
\end{align*}
\end{proof}

\begin{lemma}
\label{lem3.5}
There exists $\varphi \in \{f \in C_{0}^{\infty}([0,\infty)): f(0)=0\}$ and $\delta>0$ such that $\re \tbra{F(\mathcal{F} v_+)}{\mathcal{F}\varphi}_{j} <-\delta$ provided that $v_+ \neq 0$.  
\end{lemma}

\begin{proof}
Since $v_+ \neq0$, we may assume that $\re \mathcal{F} v_+>0$ on a set with non-zero measure for simplicity. 
Let $f\vee 0= \max\{f,0\}$.  
By Lemma \ref{rmkA.3}, there exists a real-valued function $\varphi \in \{f \in C_{0}^{\infty}([0,\infty)): f(0)=0\}$ such that $\| (\re \mathcal{F} v_+\vee 0)^{1-p} + \lambda \mathcal{F} \varphi \|_{L^{\frac{2}{1-p}}(e_j)} < \delta \| \re \mathcal{F} v_+ \|_{L^2}^{-(p+1)}$, where $\delta>0$ is defined later, noting that $\re \mathcal{F} v_+(0)=0$ by the definition of the Fourier-Sine transformation and $(\re \mathcal{F} v_+ \vee 0)^{1-p} \in L^{\frac{2}{1-p}}(e_j)$ for $0<p<1$. Then, we obtain
\begin{align*}
	\re \tbra{F(\mathcal{F} v_+)}{\mathcal{F}\varphi}_{j}
	&=- \int_{0}^{\infty} |\mathcal{F} v_+|^p (\re \mathcal{F} v_+)(\re \mathcal{F} v_+\vee 0)^{1-p}  dx 
	\\
	&\quad+ \int_{0}^{\infty} |\mathcal{F} v_+|^p (\re \mathcal{F} v_+)( (\re \mathcal{F} v_+\vee 0)^{1-p}  +\lambda\mathcal{F} \varphi) dx
	\\
	&\leq  -2\delta + \left\|(\re \mathcal{F} v_+\vee 0)^{1-p}  + \lambda\mathcal{F} \varphi \right\|_{L^{\frac{2}{1-p}}(e_j)} \norm{\mathcal{F} v_+ }_{L^2}^{p+1}
	\\
	&\leq -2\delta + \delta = -\delta,
\end{align*}
where $\delta :=\frac{1}{2} \int_{0}^{\infty} |\mathcal{F} v_+|^p (\re \mathcal{F} v_+)(\re \mathcal{F} v_+\vee 0)^{1-p}  dx  >0$ by the assumption.
Thus we get the statement. 
\end{proof}

\begin{proof}[Proof of Lemma {\ref{lem2.2}}]
We suppose that $v_+ \neq 0$ and take $\varphi$ as in Lemma \ref{lem3.5}. 
Integrating \eqref{eq3.1} on $[1,\tau]$ and taking the real part, by Lemma \ref{lem3.2}, Lemma \ref{cor3.4}, and Lemma \ref{lem3.5},  we have
\begin{align*}
	\int_{1}^{\tau} \re (i \partial_t \tbra{u}{w}_{j}) dt 
	&\cgeq -\tau^{1/2} + \int_{1}^{\tau} t^{-\frac{p}{2}} \re \tbra{F(\mathcal{F} v_+)}{\mathcal{F}\varphi}_{j} dt -C-\frac{\delta}{2}\tau^{1-\frac{p}{2}}
	\\
	&\geq -\tau^{1/2}+\frac{\delta}{2}\tau^{1-\frac{p}{2}}-C.
\end{align*}
Since $p<1$, the left hand side tends to infinity as $\tau \to \infty$. On the other hand, we have $\text{(L.H.S)} \cleq \norm{u_0}_{L^2(\mathcal{G})} \norm{\varphi}_{L^2(e_j)}$ by $L^2$-conservation law. This is a contradiction. 
\end{proof}


\subsection{Linear asymptotics}
\label{sec3.2} 

We will use the following famous lemma. 
\begin{lemma}[Kato--Kuroda--Birman (see e.g. {\cite[Theorem XI.9]{ReSiIII}})]
\label{lem3.1.1}
Let $\mathcal{H}$ be a Hilbert space and $A,B$ are self-adjoint operators on $\mathcal{H}$. If $(A+i)^{-1} - (B+i)^{-1}$ is a trace class operator, then $\Omega_{+} := \slim_{t \to \infty} e^{itA} e^{-itB} P_{ac}(B)$ exists. Namely, for any $u_0 \in \mathcal{H}$, there exists $u_{+} \in \mathcal{H}$ such that 
\begin{align*}
	\norm{e^{-itB}P_{ac}(B) u_0 - e^{-itA} u_+}_{\mathcal{H}} \to 0 \text{ as } t \to \infty. 
\end{align*}
\end{lemma}

The resolvent formula is known for the Laplacian on the star graph by \cite{KoSc06} as follows.
\begin{lemma}[{\cite[Lemma 4.2]{KoSc06}}]
Let $\lambda^2 \in \C \setminus \R$ and $\im \lambda >0$. Then $(-\Delta_M - \lambda^2)^{-1} \in B(\mathcal{H})$ is given by 
\begin{align*}
	(-\Delta_M - \lambda^2)^{-1} u(x) = \int_{\mathcal{G}} r(x,y,\lambda) u(y)dy,
	\\
	r(x,y,\lambda)= r^{0}(x,y,\lambda) + \frac{i}{2\lambda} \phi(x)G(M)\phi(y),
\end{align*}
where 
$r^{0}(x,y,\lambda):=\frac{i}{2\lambda} \diag (e^{i\lambda|x_j -y_j|})_{j=1}^{n}$, $\phi(x):=\diag(e^{i\lambda x_j})_{j=1}^{n}$, $G(M):=-(A+i\lambda B)^{-1}(A-i\lambda B)$.
\end{lemma}

Letting $G_{kl}$ be $(k,l)$-component of $G(D)-G(M)$ and $\phi_l$ be $(l,l)$-component of $\phi$ and setting $\varphi_k:=(G_{kl} \phi_{l})_{1\leq l \leq n}$, we have 
\begin{align*}
	[(-\Delta_{D} - \lambda^2)^{-1} - (-\Delta_{M} -\lambda^2)^{-1} ] u
	&=\frac{i}{2\lambda}\int_{\mathcal{G}}  \phi(x) (G(D)-G(M)) \phi(y) u(y)dy
	\\
	&=\frac{i}{2\lambda}  \phi(x) \int_{\mathcal{G}} \left(\sum_{l=1}^{n} G_{kl} \phi_l(y_l) u_l(y_l) \right)_{1\leq k \leq n} dy
	\\
	&=\frac{i}{2\lambda} \phi(x) \left( \sum_{l=1}^{n} \tbra{u_l}{G_{kl}\phi_{l}}_{l} \right)_{1\leq k \leq n} 
	\\
	&=\frac{i}{2\lambda}  \sum_{l=1}^{n}  \left(\tbra{u_l}{G_{kl}\phi_{l}}_{l} \phi_k(x_k) \right)_{1\leq k \leq n} 
	\\
	&=\frac{i}{2\lambda}   \left( \tbra{u}{\varphi_k} \phi_k(x_k) \right)_{1\leq k \leq n}.
\end{align*}
This means that $(-\Delta_{D} - i)^{-1} - (-\Delta_{M} -i)^{-1}$ (as $\lambda=(1+i)/\sqrt{2}$) is finite rank. And thus, it also belongs to the trace class. Applying Lemma \ref{lem3.1.1} as $\mathcal{H}=L^2(\mathcal{G})$, we complete the proof of Lemma \ref{lem2.3}.


\subsection{Proof of main result}
\label{sec3.3}

We give the proof of main result. 

\begin{proof}[Proof of Theorem {\ref{thm2.1}}]
By Lemma \ref{lem2.3}, 
\begin{align*}
	\norm{u(t) - (e^{it\Delta_{K}} u_{+}+l(t))}_{L^2(\mathcal{G})} \to 0 \quad (t \to \infty)
\end{align*}
implies that there exists $v_{+} \in L^2(\mathcal{G})$ such that
\begin{align*}
	\norm{u(t) - (e^{it\Delta_{D}} v_+ + l(t))}_{L^2(e_j)} \to 0 \quad (t \to \infty)
\end{align*}
for any $j =1,...,n$. 
By Lemma \ref{lem2.2}, we get $v_+\equiv 0$ on $\mathcal{G}$. This means that $\norm{e^{it\Delta_{K}} v_{+}}_{L^2(\mathcal{G})} \to 0$ from Lemma \ref{lem2.2}. However, the $L^2$-conservation law shows $ v_{+}\equiv 0$. 
\end{proof}


\appendix
\section{Weak solution}
In the appendix, we discuss that an $L^2$-solution is a weak solution. 
We say that a pair $(q,r)$ is admissible if $2/q = 1/2-1/r$ and $2\leq q,r \leq \infty$. 
We define the $L^2$-solution (or the Strichartz class solution) as follows.
\begin{definition}
Let $I$ be a time interval containing $0$. 
We say that $u$ is an $L^2$-solution (or a Strichartz class solution) to \eqref{NLS} on $I$ if $u \in C(I:L^2(\mathcal{G})) \cap \bigcap_{(q,r):\text{admissible}} L^q(I:L^r(\mathcal{G}))$ and $u$ satisfies
\begin{align*}
	u(t)=e^{it\Delta_{M}} u_0 + i \lambda \int_{0}^{t} e^{i(t-s)\Delta_{M}} (|u|^p u)(s) ds
\end{align*}
for all $t \in I$ and almost all $x \in \mathcal{G}$.
\end{definition}

A unique $L^2$-solution to \eqref{NLS} exists if $u_0 \in L^2(\mathcal{G})$ by Grecu and Ignat \cite{GrIg19}. They also showed $L^2$-conservation law, i.e., $\norm{u(t)}_{L^2(\mathcal{G})}=\norm{u_0}_{L^2(\mathcal{G})}$, and thus the solution is global. 
The $L^2$-solution is a weak solution in the following sense. 

\begin{lemma}
\label{lemA.1}
If $u$ is an $L^2$-solution and $\varphi \in \{f \in H^2((0,\infty)) \cap C([0,\infty)) : f(0)=0\}$, we have
\begin{align}
\label{eqA.1}
	i \tbra{u(\tau)}{\varphi}_{j} - i \tbra{u_0}{\varphi}_{j} + \int_{0}^{\tau} \tbra{u}{\varphi''}_{j} dt + \int_{0}^{\tau} u_j(t,0+) \overline{\partial_x \varphi (0+)}  dt 
	\\=-\lambda \int_{0}^{\tau} \tbra{|u|^p u}{\varphi}_{j} dt.
\notag
\end{align}
\end{lemma}

\begin{proof}
Let $H=-\Delta_{M}$. We define $u_\varepsilon:=(I+\varepsilon^2 H )^{-1}u$, where $u$ is the $L^2$-solution. Then, from the argument in \cite{GrIg19}, it holds that $u_{\varepsilon} \in C(I:\mathscr{D}(H)) \cap W^{1,1,}(I:L^2(\mathcal{G}))$ and 
\begin{align*}
\begin{cases}
	i \partial_t u_{\varepsilon} = Hu_{\varepsilon} - (I+\varepsilon^2 H)^{-1} F(u),
	\\
	u_{\varepsilon}(0)=(I+\varepsilon^2 H )^{-1}u_{0},
\end{cases}
\end{align*}
where $F(u)=\lambda |u|^p u$ and $(I+\varepsilon^2 H )^{-1}F(u) \in L^1(I:\mathscr{D}(H))$. Multiplying the equation by the complex conjugate of $\varphi \in  \{f \in H^2((0,\infty)) \cap C([0,\infty)) : f(0)=0\}$ and integrating on an edge $e_j$, it follows from the integration by parts that
\begin{align*}
	i \tbra{\partial_t u_{\varepsilon}}{\varphi}_{j}  =- \tbra{u_{\varepsilon}}{\varphi''}_{j} +u_{\varepsilon j}(t,0+) \overline{\partial_x \varphi (0+)} - \tbra{(I+\varepsilon^2 H)^{-1} F(u)}{\varphi}_{j}.
\end{align*}
Integrating this on $[0,\tau)$ for $\tau \in I$, we obtain
\begin{align*}
	i \tbra{u_{\varepsilon}(\tau)}{\varphi}_{j} - \tbra{u_{\varepsilon}(0)}{\varphi}_{j}   
	&= -\int_{0}^{\tau}  \tbra{u_{\varepsilon}}{\varphi''}_{j}dt
	 +\int_{0}^{\tau}  u_{\varepsilon j}(t,0+) \overline{\partial_x \varphi (0+)} dt 
	\\
	&\quad- \int_{0}^{\tau}  \tbra{(I+\varepsilon^2 H)^{-1} F(u)}{\varphi}_{j} dt.
\end{align*}
By the argument in \cite{GrIg19}, taking $\varepsilon\to 0$, we have \eqref{eqA.1}.
\end{proof}

From this lemma, we get the following lemma, which is one of keys to show Theorem \ref{thm2.1} (see Section \ref{sec2.3}). 

\begin{lemma}
Let $u$ be an $L^2$-solution and $w=e^{it\Delta_{D}}\varphi$, where $\varphi \in \{f \in H^2(\mathcal{G}) : f(0)=0\}$. Then we have
\begin{align*}
	i \tbra{u(\tau)}{w(\tau)}_{j} - i \tbra{u_0}{\varphi}_{j}+ \int_{0}^{\tau} u_j(t,0+) \overline{\partial_x w (t,0+)}  dt =-\lambda \int_{0}^{\tau} \tbra{|u|^p u}{w}_{j} dt.
\end{align*}
\end{lemma}

\begin{proof}
Multiplying the equation by the complex conjugate of $w$ and integrating on an edge $e_j$, it follows from the argument in the proof of Lemma \ref{lemA.1} that
\begin{align*}
	i \tbra{\partial_t u_{\varepsilon}}{w}_{j}  = -\tbra{u_{\varepsilon}}{w''}_{j} +u_{\varepsilon j}(t,0+) \overline{\partial_x w (t,0+)} - \tbra{(I+\varepsilon^2 H)^{-1} F(u)}{w}_{j},
\end{align*}
where $u_{\varepsilon}$ is as in the proof and note that $w_j (t) \in \{f \in H^2((0,\infty)) \cap C([0,\infty)) : f(0)=0\}$. Since $w$ is a solution of 
\begin{align*}
	i\partial_t w + w''=0 \text{ and } w(t,0)=0,
\end{align*}
we obtain
\begin{align*}
	i \partial_t \tbra{ u_{\varepsilon}}{w}_{j} = u_{\varepsilon j}(t,0+) \overline{\partial_x w (t,0+)} - \tbra{(I+\varepsilon^2 H)^{-1} F(u)}{w}_{j}.
\end{align*}
Integrating this on $[0,\tau)$ for $\tau \in I$ and taking $\varepsilon\to 0$, this completes the proof. 
\end{proof}

\begin{acknowledgement}
The authors would like to express deep appreciation to Dr. Tomoyuki Tanaka for introducing the papers related to NLS on the star graph.
The second author is supported by JSPS KAKENHI Grant-in-Aid for Early-Career Scientists JP18K13444 and the third author is supported by JSPS KAKENHI  Grant-in-Aid for Young Scientists (B) JP17K14218 and, partially, for Scientific Research (B) JP17H02854. 
\end{acknowledgement}



\begin{thebibliography}{99}

\bibitem{ACFN11}
Riccardo Adami, Claudio Cacciapuoti, Domenico Finco, Diego Noja, 
\textit{Fast solitons on star graphs},
Rev. Math. Phys. {\bf 23} (2011), no. 4, 409--451.


\bibitem{ACFN14}
Riccardo Adami, Claudio Cacciapuoti, Domenico Finco, Diego Noja, 
\textit{Variational properties and orbital stability of standing waves for NLS equation on a star graph}, 
J. Differential Equations {\bf 257} (2014), no. 10, 3738--3777.

\bibitem{AnGo18}
Jaime Angulo Pava, Nataliia  Goloshchapova, 
\textit{Extension theory approach in the stability of the standing waves for the NLS equation with point interactions on a star graph}, Adv. Differential Equations {\bf 23} (2018), no. 11-12, 793--846.


\bibitem{Bar84}
Jacqueline E. Barab, 
\textit{Nonexistence of asymptotically free solutions for a nonlinear Schr\"{o}dinger equation}, 
J. Math. Phys. {\bf 25} (1984), no. 11, 3270--3273.


\bibitem{Caz03} Thierry Cazenave, 
\textit{Semilinear Schr\"{o}dinger Equations}, 
Courant Lecture Notes in Mathematics, vol. 10, American Mathematical Society, Courant Institute of Mathematical Sciences, 2003.

\bibitem{EsKaHa19}
Liliana Esquivel, Nakao Hayashi, Elena I. Kaikina,  
\textit{Inhomogeneous Dirichlet-boundary value problem for one dimensional nonlinear Schr\"{o}dinger equations via factorization techniques}, 
J. Differential Equations {\bf 266} (2019), no. 2-3, 1121--1152.

\bibitem{GiOz93}
Jean Ginibre, Tohru Ozawa, 
\textit{Long range scattering for nonlinear Schr\"{o}dinger and Hartree equations in space dimension $n\geq 2$}, Comm. Math. Phys. {\bf 151} (1993), no. 3, 619--645.

\bibitem{GoOh19}
Nataliia Goloshchapova, Masahito Ohta,
\textit{Blow-up and strong instability of standing waves for the NLS-$\delta$ equation on a star graph},
preprint, arXiv:1908.07122.

\bibitem{GrIg19} 
Andreea Grecu, Liviu I. Ignat, 
\textit{The Schr\"{o}dinger equation on a star-shaped graph under general coupling conditions},
J. Phys. A {\bf 52} (2019), no. 3, 035202, 26 pp.

\bibitem{HaNa98}
Nakao Hayashi, Pavel I. Naumkin,  
\textit{Asymptotics for large time of solutions to the nonlinear Schr\"{o}dinger and Hartree equations}, 
Amer. J. Math. {\bf 120} (1998), no. 2, 369--389.

\bibitem{Kai19}
Adilbek Kairzhan,  
\textit{Orbital instability of standing waves for NLS equation on star graphs}, 
Proc. Amer. Math. Soc. {\bf 147} (2019), no. 7, 2911--2924.

\bibitem{KoSc06}
Vadim Kostrykin, Robert  Schrader, 
\textit{Laplacians on metric graphs: eigenvalues, resolvents and semigroups},
Quantum graphs and their applications, 201--225, Contemp. Math., {\bf 415}, Amer. Math. Soc., Providence, RI, 2006.


\bibitem{MaMuSe17}
Satoshi Masaki, Jason Murphy, Jun-ichi Segata,
Modified scattering for the 1d cubic NLS with a repulsive delta potential,
preprint, arXiv:1708.00392.

\bibitem{MuNa19}
Jason Murphy, Kenji Nakanishi,
\textit{Failure of scattering to solitary waves for long-range nonlinear Schr\"{o}dinger equations}. 
preprint, arXiv:1906.01802. 


\bibitem{Oza91}
Tohru Ozawa, 
\textit{Long range scattering for nonlinear Schr\"{o}dinger equations in one space dimension}, 
Comm. Math. Phys. {\bf 139} (1991), no. 3, 479--493. 

\bibitem{ReSiIII}
Michael Reed, Barry Simon,
``Methods of modern mathematical physics. III. Scattering theory", 
Academic Press [Harcourt Brace Jovanovich, Publishers], New York-London, 1979. xv+463 pp.

\bibitem{Seg15}
Jun-Ichi Segata, 
\textit{Final state problem for the cubic nonlinear Schr\"{o}dinger equation with repulsive delta potential},
Comm. Partial Differential Equations {\bf 40} (2015), no. 2, 309--328.

\bibitem{Str74}
Walter. A. Strauss, 
\textit{Nonlinear scattering theory}, in “Scattering theory in mathematical physics”,
53--78, Reidel, Dordrecht, 1974.

\bibitem{TsYa84}
Yoshio Tsutsumi, Kenji Yajima, 
\textit{The asymptotic behavior of nonlinear Schr\"{o}dinger equations}, 
Bull. Amer. Math. Soc. (N.S.) {\bf 11} (1984), no. 1, 186--188.

\bibitem{Yos18}
Kouki Yoshinaga, Master Thesis, Graduate School of Information Science and Technology, Osaka University, (2018), written in Japanese. 


\end{thebibliography}
\end{document}